\documentclass[12pt]{amsart}
\usepackage{mlmodern}

\usepackage[utf8]{inputenc}
\usepackage[margin=1.2in]{geometry}
\usepackage{amsmath,amssymb}
\usepackage{amsfonts}
\usepackage{mathtools}
\usepackage{float}
\usepackage{amsthm}
\usepackage{mathrsfs}
\usepackage{graphicx}
\usepackage[colorlinks=true, linkcolor=blue, citecolor=blue, urlcolor=blue, breaklinks=true]{hyperref}
\usepackage{thmtools}
\usepackage[capitalise]{cleveref}
\usepackage{tikz}
\usepackage{soul}
\usepackage{enumitem}
\usepackage[symbol]{footmisc}

\newcommand\R{\mathbb{R}}
\newcommand\Z{\mathbb{Z}}

\newcommand\T{\mathcal{T}}
\newcommand\D{\mathcal{D}}
\renewcommand\O{\mathcal{O}}

\newtheorem{theorem}{Theorem}[section]

\newtheorem{definition}[theorem]{Definition}
\newtheorem{example}[theorem]{Example}

\newtheorem{proposition}[theorem]{Proposition}
\newtheorem{remark}[theorem]{Remark}

\title{Gamma-positivity for octopuses: a bijective proof}

\author{Krishna Menon}
\address{\scriptsize Department of Mathematics, KTH Royal Institute of Technology, Stockholm, Sweden}
\email{puzhan@kth.se}

\begin{document}

\begin{abstract}
    Chapoton introduced an interesting family of polytopes called arbor polytopes, where a polytope $\mathcal{Q}_\tau$ is associated to any arbor $\tau$. 
    Chapoton conjectured that the polynomial $h(\tau)$ which counts lattice points in $\mathcal{Q}_\tau$ by number of nonzero entries is palindromic and unimodal. 
    Athanasiadis, Xiao, and Yan recently proved that $h(\tau)$ is gamma-positive for a certain class of arbors they call octopuses. 
    Since their proof was computational, they asked for one that is bijective. 
    We present such a proof. 
    We also extend their results by showing that $h(\tau)$ is gamma-positive for a larger class of arbors we call lopsided octopuses.
\end{abstract}

\maketitle

\section{Preliminaries}

An \emph{arbor} of size $n$ is a rooted tree with vertices the blocks of a partition of $[n]$. 
For any vertex $v$ of an arbor, we use $\D(v)$ to denote the union of all the descendants of $v$, including $v$ itself. 
Given an arbor $\tau$ of size $n$, the polytope $\mathcal{Q}_\tau$ in $\R^n$ is defined by the inequalities
\begin{align*}
    x_i &\geq 0 \text{ for all }i \in [n]\text{ and }\\[0.2cm]
    \sum_{i \in \D(v)} x_i &\leq |\D(v)| \text{ for all vertices $v$ of $\tau$}.
\end{align*}
These polytopes were introduced and studied by Chapoton \cite{chap}. 
Special classes of these polytopes have been studied as well \cite{comp,oct}. 
A generalization of these polytopes (by considering preorders instead of arbors) has also been studied by Athanasiadis and Chapoton \cite{preord}.

The polynomials we study, which were introduced by Chapoton \cite{chap}, are those that count the lattice points in these polytopes by number of nonzero coordinates. 
For any arbor $\tau$, the polynomial $h(\tau)$ is defined by
\begin{equation*}
    h(\tau, t) = \sum_x t^{|\{i \in [n] \colon x_i \neq 0\}|}
\end{equation*}
where the sum is over all $x \in \mathcal{Q_\tau} \cap \Z^n$. 
These polynomials are conjectured to be palindromic and unimodal \cite[Conjecture 0.1]{chap}. 
For an arbor $\tau$ of size $n$, the polynomial $h(\tau)$ is said to be \emph{$\gamma$-positive} if there are nonnegative numbers $\gamma_j(\tau)$ for which
\begin{equation*}
    h(\tau, t) = \sum_{j = 0}^{\lfloor n/2 \rfloor} \gamma_j(\tau) t^j(1 + t)^{n - 2j}.
\end{equation*}
This property would, in particular, prove that $h(\tau)$ is palindromic and unimodal. 
See \cite{athsurv} for a survey on $\gamma$-positivity.

Athanasiadis, Xiao, and Yan \cite{oct} studied arbor polytopes for a class of arbors they call \emph{octopuses}. 
For $0 \leq k \leq n$, the octopus $\tau_{n, k}$ has the $(n - k)$-element set $[k + 1, n]$ as its root and the singletons $\{i\}$ for $i \in [k]$ as leaves of the root. 
They prove the following result.

\begin{proposition}[{\cite[Proposition 2.2]{oct}}]\label{prop}
    For any $n, k$, we have
    \begin{equation*}
        h(\tau_{n, k}, t) = \sum_{j = 0}^{\lfloor n/2 \rfloor} \binom{n - k}{j} \binom{n - j}{j} t^j(1 + t)^{n - 2j}.
    \end{equation*}
\end{proposition}

Hence, Athanasiadis, Xiao, and Yan prove that the polynomials $h(\tau_{n, k})$ are $\gamma$-positive, and explicitly compute their coefficients in the $\gamma$-basis.
However, their proof is rather computational. 
The first result in this article, presented in \Cref{proofsec}, is a bijective proof of the above proposition.

We extend this result to a more general class of arbors that we call \emph{lopsided octopuses}. 
A lopsided octopus is an octopus where we allow leaves to have cardinality $2$ (see \Cref{fig:lop}). 
In \Cref{lopsec}, we present a bijective proof that $h(\tau)$ is $\gamma$-positive for lopsided octopuses and give an expression for its $\gamma$-coefficients.

\section{Combinatorial proof of \Cref{prop}}\label{proofsec}

Throughout this section, we let $k, n$ be integers such that $0 \leq k \leq n$. 
We use $\T_{n, k}$ to denote $\mathcal{Q}_{\tau_{n, k}} \cap \Z^n$. 
That is, $\T_{n, k}$ is the set of integer-valued tuples $x = (x_1, x_2, \ldots, x_n)$ such that
\begin{align*}
    &x_i \geq 0 \text{ for } i \in [n],\\
    &x_i \leq 1 \text{ for } i \in [k], \text{ and}\\
    &x_1 + x_2 + \cdots + x_n \leq n.
\end{align*}
To prove the result combinatorially, we view $h(\tau_{n, k})$ as the generating function of objects consisting of dots, numbers, and circles.

\begin{definition}
    The set $\O_{n, k}$ consists of all objects of the following form.
    \begin{itemize}
        \item There are $n$ dots in a row and they are labeled using the numbers $1, 2, \ldots, n$ in order. 
        The labels are written below the dots.
        \item Each dot can be circled using either a black or a red circle.
        \item All dots circled black must have a label in $[k]$.
        \item Some of the labels from $[k + 1, n]$ can be circled using a red circle and the number of such labels must be equal to the number of dots circled red.
    \end{itemize}
\end{definition}

When we draw an object in $\O_{n, k}$, we draw a dotted line to separate the labels in $[k]$ from those in $[k + 1, n]$ (see the example below). 
Hence, the only labels that can be circled (necessarily in red) are after the dotted line.

\begin{example}
    An object in $\O_{7, 4}$ is shown below.
    
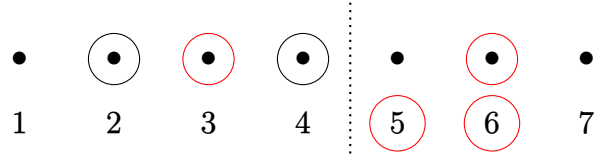
\begin{center}
        \begin{tikzpicture}[xscale = 1.25]
            \node at (1, 0) {$\bullet$};
            \node[draw = black, circle] at (2, 0) {$\bullet$};
            \node[draw = red, circle] at (3, 0) {$\bullet$};
            \node[draw = black, circle] at (4, 0) {$\bullet$};
            \node at (5, 0) {$\bullet$};
            \node[draw = red, circle] at (6, 0) {$\bullet$};
            \node at (7, 0) {$\bullet$};

            \foreach \x in {1,2,3,4,7}
            {
            \node at (\x, -0.85) {$\x$};
            }
            \node[draw = red, circle] at (5, -0.85) {$5$};
            \node[draw = red, circle] at (6, -0.85) {$6$};

            \draw[thick,dotted] (4.5, -1.25) -- (4.5, 0.75);
        \end{tikzpicture}
    \end{center}
\end{example}

We now exhibit a bijection between $\T_{n, k}$ and $\O_{n, k}$. 
Given a tuple $x \in \T_{n, k}$, we construct an object in $\O_{n, k}$ as follows.
\begin{itemize}
    \item For any $i \in [k]$, if $x_i = 1$, we circle in black the dot labeled $i$.
    \item For any $i \in [k + 1, n]$, if $x_i \neq 0$, we circle in red the \emph{label} $i$. 
    If $i_1$ is the smallest such index and $x_{i_1} = k_1$, then we circle in red the $k_1^{th}$ uncircled dot (read from left to right). 
    If $i_2$ is the second smallest such index and $x_{i_2} = k_2$, then we circle in red the $k_2^{th}$ uncircled dot after the last dot circled red. 
    We continue this process for all such indices.
\end{itemize}
It can be verified that this is a bijection.

\begin{example}\label{bijecex}
    For $n = 9$ and $k = 4$, we have $x = (0,1,1,0,1,0,3,3,0) \in \T_{9, 4}$. 
    The corresponding object $O$ in $\O_{9, 4}$ is shown below.
    
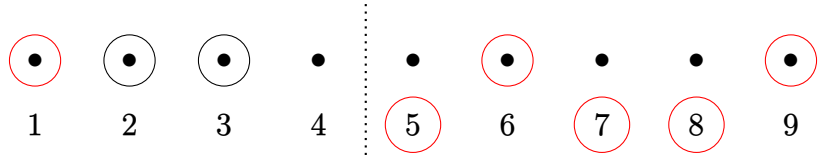
\begin{center}
        \begin{tikzpicture}[xscale = 1.25]
            \node[draw = red, circle] at (1, 0) {$\bullet$};
            \node[draw = black, circle] at (2, 0) {$\bullet$};
            \node[draw = black, circle] at (3, 0) {$\bullet$};
            \node at (4, 0) {$\bullet$};
            \node at (5, 0) {$\bullet$};
            \node[draw = red, circle] at (6, 0) {$\bullet$};
            \node at (7, 0) {$\bullet$};
            \node at (8, 0) {$\bullet$};
            \node[draw = red, circle] at (9, 0) {$\bullet$};

            \foreach \x in {1,2,3,4,6,9}
            {
            \node at (\x, -0.85) {$\x$};
            }
            \node[draw = red, circle] at (5, -0.85) {$5$};
            \node[draw = red, circle] at (7, -0.85) {$7$};
            \node[draw = red, circle] at (8, -0.85) {$8$};

            \draw[thick,dotted] (4.5, -1.25) -- (4.5, 0.75);
        \end{tikzpicture}
    \end{center}
\end{example}

It might be easier to view this bijection by converting $x$ into a partition $y$ via its partial sums. 
After doing this, the bijection can be viewed as using the objects in $\O_{n, k}$ to record the coordinates of the corner boxes of the Young diagram of $y$. 
This idea is elucidated in the following example.

\begin{example}\label{youngex}
    For $n, k, x, O$ as in \Cref{bijecex}, the corresponding partition of partial sums is $y = (0,1,2,2,3,3,6,9,9)$. 
    The Young diagram of this partition in the $n \times n$ grid is shown in \Cref{fig:young}. 
    To specify the coordinates of the corner boxes, note that in the first $k$ rows (corner boxes marked with a black dot) we only need to specify the column number. 
    These are specified in $O$ by the dots circled black.
    For the remaining corner boxes (marked with red dots), the column numbers are specified with the labels circled red in $O$. 
    The row numbers are specified by the dots circled red; since there are two dots circled in black, we represent the columns $3, 4, \ldots, 9$ by the remaining dots from left to right.
\end{example}

\begin{figure}[ht]
    \centering
    \begin{tikzpicture}[scale = 0.85]
        \draw[fill, cyan!10] (0, 9) rectangle (9, 7);
        \draw[fill, cyan!10] (0, 7) rectangle (6, 6);
        \draw[fill, cyan!10] (0, 6) rectangle (3, 4);
        \draw[fill, cyan!10] (0, 4) rectangle (2, 2);
        \draw[fill, cyan!10] (0,2) rectangle (1, 1);
        
        \draw[thin, dotted] (0, 0) grid (9, 9);
        
        \foreach \x in {1,2,4,5,7,8}
        {
        \node at (\x - 0.5, -0.75) {$\x$};
        }
        \foreach \x in {3,6,9}
        {
        \node[draw = red, circle] at (\x - 0.5, -0.75) {$\x$};
        }

        \foreach \x in {1,4,6,9}
        {
        \node at (-0.75, \x - 0.5) {$\x$};
        }
        \foreach \x in {2,3}
        {
        \node[draw, circle] at (-0.75, \x - 0.5) {$\x$};
        }
        \foreach \x in {5,7,8}
        {
        \node[draw = red, circle] at (-0.75, \x - 0.5) {$\x$};
        }
        
        \draw[thick] (0, 0) -- (0, 1) -- (1, 1) -- (1, 2) -- (2, 2) -- (2, 4) -- (3, 4) -- (3, 6) -- (6, 6) -- (6, 7) -- (9, 7) -- (9, 9);

        \node at (1 - 0.5, 2 - 0.5) {$\bullet$};
        \node at (2 - 0.5, 3 - 0.5) {$\bullet$};
        \node at (3 - 0.5, 5 - 0.5) {\color{red} $\bullet$};
        \node at (6 - 0.5, 7 - 0.5) {\color{red} $\bullet$};
        \node at (9 - 0.5, 8 - 0.5) {\color{red} $\bullet$};
    \end{tikzpicture}
    \caption{Young diagram of the partition in \Cref{youngex}.}
    \label{fig:young}
\end{figure}
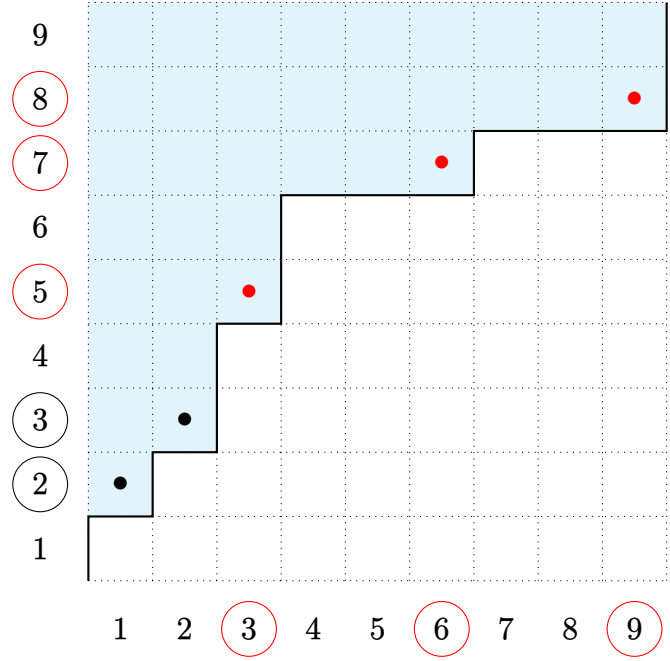

Under the bijection presented above, the number of nonzero entries in the tuple $x \in \T_{n, k}$ is the number of circled dots in the corresponding object in $\O_{n, k}$. 
We now use a similar idea to \cite[Proposition 2.2]{comp} to prove the result.

We say that an object in $\O_{n, k}$ is \emph{primitive} if
\begin{itemize}
    \item there are no dots circled black and
    \item there is no $i \in [k + 1, n]$ such that both the label $i$ as well as its corresponding dot are circled red.
\end{itemize}
To any object in $\O_{n, k}$, we can associate a primitive object by removing all black circles and if there are two red circles on the same vertical line, remove both red circles.

\begin{example}
    The primitive object associated to the object in $\O_{8, 3}$ given by
    
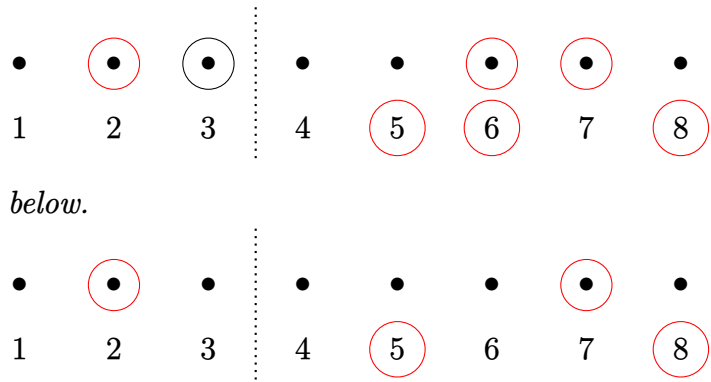
\begin{center}
        \begin{tikzpicture}[xscale = 1.25]
            \node at (1, 0) {$\bullet$};
            \node[draw = red, circle] at (2, 0) {$\bullet$};
            \node[draw = black, circle] at (3, 0) {$\bullet$};
            \node at (4, 0) {$\bullet$};
            \node at (5, 0) {$\bullet$};
            \node[draw = red, circle] at (6, 0) {$\bullet$};
            \node[draw = red, circle] at (7, 0) {$\bullet$};
            \node at (8, 0) {$\bullet$};

            \foreach \x in {1,2,3,4,7}
            {
            \node at (\x, -0.85) {$\x$};
            }
            \node[draw = red, circle] at (5, -0.85) {$5$};
            \node[draw = red, circle] at (6, -0.85) {$6$};
            \node[draw = red, circle] at (8, -0.85) {$8$};

            \draw[thick,dotted] (3.5, -1.25) -- (3.5, 0.75);
        \end{tikzpicture}
    \end{center}
    is the one shown below.
    \begin{center}
        \begin{tikzpicture}[xscale = 1.25]
            \node at (1, 0) {$\bullet$};
            \node[draw = red, circle] at (2, 0) {$\bullet$};
            \node at (3, 0) {$\bullet$};
            \node at (4, 0) {$\bullet$};
            \node at (5, 0) {$\bullet$};
            \node at (6, 0) {$\bullet$};
            \node[draw = red, circle] at (7, 0) {$\bullet$};
            \node at (8, 0) {$\bullet$};

            \foreach \x in {1,2,3,4,7}
            {
            \node at (\x, -0.85) {$\x$};
            }
            \node[draw = red, circle] at (5, -0.85) {$5$};
            \node at (6, -0.85) {$6$};
            \node[draw = red, circle] at (8, -0.85) {$8$};

            \draw[thick,dotted] (3.5, -1.25) -- (3.5, 0.75);
        \end{tikzpicture}
    \end{center}
\end{example}

We can construct all objects in $\O_{n, k}$ that are associated to a given primitive object $O$ as follows. 
Suppose that in $O$, the labels of the dots that are circled red are $S = \{s_1, \ldots, s_j\}$ and the labels that are circled red are $T = \{t_1, \ldots, t_j\}$. 
By definition $S \cap T = \varnothing$. 
For any $i \in [n] \setminus (S \cup T)$, we can modify $O$ by either
\begin{itemize}
    \item circling in black the dot labeled $i$, if $i \in [k]$, or
    \item circling in red both the dot labeled $i$ as well as the label $i$, if $i \in [k + 1, n]$.
\end{itemize}
Each such modification increases the number of circled dots by $1$ and there are precisely $n - 2j$ such modifications possible. 
It is not difficult to see that the objects in $\O_{n, k}$ that are associated to the primitive object $O$ are precisely those built by modifying $O$ using the operations mentioned above.

This gives us the required result since the primitive objects in $\O_{n, k}$ with $j$ circled dots can be constructed by
\begin{itemize}
    \item first choosing the labels to be circled red, which can be done in $\binom{n - k}{j}$ ways, and then
    \item choosing the dots to be circled red. 
    Since the object must be primitive, there are $\binom{n - j}{j}$ ways to choose these dots.
\end{itemize}
This gives us the required expression
\begin{equation*}
    h(\tau_{n, k}, t) = \sum_{j = 0}^{\lfloor n/2 \rfloor} \binom{n - k}{j} \binom{n - j}{j} t^j(1 + t)^{n - 2j}.
\end{equation*}

\begin{remark}
    The same idea can be used to prove $\gamma$-positivity of the corresponding polynomial for the generalized polytopes $\mathcal{Q}_{n, d, k}$ studied in \cite[Section 5]{oct}. 
    We would consider objects where there are $n + d$ dots but only the first $n$ are labeled. 
    The bijection and proof of $\gamma$-positivity then follows just as presented above.
\end{remark}

\section{Lopsided octopuses}\label{lopsec}

Let $n, k_1, k_2$ be nonnegative integers such that $k = k_1 + 2k_2 \leq n$. 
The lopsided octopus $\tau_{n, k_1, k_2}$ is the arbor with root $[k + 1, n]$ and singletons $\{i\}$ for $i \in [k_1]$ and doubletons $\{i, i + 1\}$ for $i \in \{k_1 + 1, k_1 + 3, \ldots, k_1 + 2k_2 - 1\}$ as leaves attached to the root.\footnote[1]{One can imagine an \href{https://ansatsukyoshitsu.fandom.com/wiki/Korosensei}{\color{black}octopus} that skips $k_1$ legs during leg day.} 
Taking $k_2 = 0$ gives usual octopuses.

\begin{example}
    The lopsided octopus $\tau_{10, 3, 2}$ is shown in \Cref{fig:lop}.
\end{example}

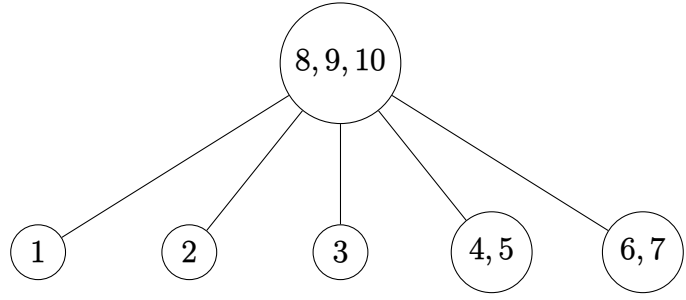
\begin{figure}[ht]
    \centering
    \begin{tikzpicture}[scale = 2,yscale=1.25]
        \node[draw,circle] (r) at (2, 1) {$8, 9, 10$};
        \node[draw,circle] (l1) at (0, 0) {$1$};
        \node[draw,circle] (l2) at (1, 0) {$2$};
        \node[draw,circle] (l3) at (2, 0) {$3$};
        \node[draw,circle] (l4) at (3, 0) {$4, 5$};
        \node[draw,circle] (l5) at (4, 0) {$6, 7$};

        \draw (l1) -- (r) -- (l2);
        \draw (l3) -- (r) -- (l4);
        \draw (r) -- (l5);
    \end{tikzpicture}
    \caption{The lopsided octopus $\tau_{10,3,2}$.}
    \label{fig:lop}
\end{figure}

We extend the proof in the previous section to show that $h(\tau_{n, k_1, k_2})$ is $\gamma$-positive. 
Just as before, we use $\T_{n, k_1, k_2}$ to denote $\mathcal{Q}_{\tau_{n, k_1, k_2}} \cap \Z^n$. 
That is, $\T_{n, k_1, k_2}$ consists of the subset of $\T_{n, k_1}$ that also satisfies
\begin{equation*}
    x_i + x_{i + 1} \leq 2 \text{ for } i \in \{k_1 + 1, k_1 + 3, \ldots, k_1 + 2k_2 - 1\}.
\end{equation*}
To construct the objects $\O_{n, k_1, k_2}$, we use almost the same method as in the previous section, except when $x_i = 2$ for some $i \in [k_1 + 1, k_1 + 2k_2]$. 
In such a case, if $i \in \{k_1 + 2j - 1, k_1 + 2j\}$, then, as before, we circle in black the dot labeled $i$, but we also enclose the dots corresponding to $k_1 + 2j - 1$ and $k_1 + 2j$ in a box. 
Note that a box can have at most one circle inside it. 
This will be elucidated in the following example.

\begin{example}
    For $n = 10, k_1 = 3, k_2 = 2,$ and $x = (1,0,0,2,0,1,0,2,0,3) \in \T_{10, 3, 2}$, the corresponding object in $\O_{n, k_1, k_2}$ is shown below. 
    Note that the uncircled dot inside the box is ignored when we start placing red circles around dots.
    
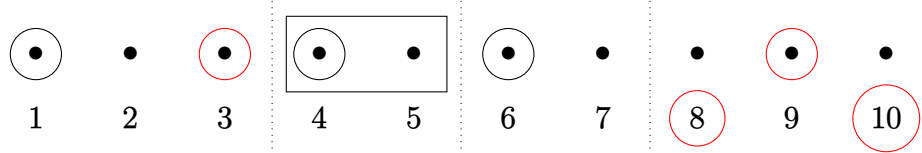
\begin{center}
        \begin{tikzpicture}[xscale = 1.25]
            \foreach \x in {2,5,7,8,10}
            {
            \node at (\x, 0) {$\bullet$};
            }
            \foreach \x in {1,4,6}
            {
            \node[draw, circle] at (\x, 0) {$\bullet$};
            }
            \foreach \x in {3,9}
            {
            \node[draw = red, circle] at (\x, 0) {$\bullet$};
            }

            \foreach \x in {1,2,3,4,5,6,7,9}
            {
            \node at (\x, -0.85) {$\x$};
            }
            \foreach \x in {8,10}
            {
            \node[draw = red, circle] at (\x, -0.85) {$\x$};
            }

            \draw (3.65, 0.5) rectangle (5.35, -0.5);

            \draw[dotted] (3.5, -1.25) -- (3.5, 0.75);
            \draw[dotted] (5.5, -1.25) -- (5.5, 0.75);
            \draw[dotted] (7.5, -1.25) -- (7.5, 0.75);
        \end{tikzpicture}
    \end{center}
    A dotted line is drawn after $k_1$ and after the pair $\{k_1 + 2j- 1, k_1 + 2j\}$ for each $j \in [k_2]$.
\end{example}

Examining this bijection gives the following definition for the objects $\O_{n, k_1, k_2}$.

\begin{definition}
    The set $\O_{n, k_1, k_2}$ consists of the objects in $\O_{n, k_1 + 2k_2}$ where we are allowed to place boxes as follows: 
    A box can be put around the dots labeled by $k_1 + 2j - 1$ and $k_1 + 2j$ for some $j \in [k_2]$ if exactly one of these dots is circled black, and the other is uncircled.
\end{definition}

We say that an object in $\O_{n, k_1, k_2}$ is \emph{primitive} if it has no black circle that is outside a box, and it has no dot circled red whose label is also circled red. 
Just as before, these primitive objects will correspond to the $\gamma$-coefficients. 
To any object in $\O_{n, k_1, k_2}$, we can associate a primitive object by deleting all black circles that are not in a box, as well as any pair of red circles that are on the same vertical line. 
And we can reverse these operations to construct all objects that are associated to a given primitive object.

Using the same ideas as in the previous section, setting $\tau = \tau_{n, k_1, k_2}$, this gives us
\begin{equation*}
    h(\tau, t) = \sum_{j = 0}^{\lfloor n/2 \rfloor} \gamma_j(\tau) t^j(1 + t)^{n - 2j}
\end{equation*}
where $\gamma_j(\tau)$ is the number of primitive objects in $\O_{n, k_1, k_2}$ that have $j$ dots that are circled. 
Counting such primitive objects gives us the following result.

\begin{proposition}
    Let $n, k_1, k_2$ be such that $k = k_1 + 2k_2 \leq n$ and $\tau = \tau_{n, k_1, k_2}$. 
    The polynomial $h(\tau)$ is $\gamma$-positive and for any $j$, we have
    \begin{equation*}
        \gamma_j(\tau) = \sum_{i = 0}^{k_2} 2^i\binom{k_2}{i} \binom{n - k}{j - i}\binom{n - j - i}{j - i}.
    \end{equation*}
\end{proposition}

\begin{proof}
    As mentioned above, $\gamma_j(\tau)$ is the number of primitive objects in $\O_{n, k_1, k_2}$ that have exactly $j$ dots that are circled. 
    The expression above is split based on the number $i$ of boxes in such an object. 
    There are $2^i\binom{k_2}{i}$ ways to place $i$ boxes and choose which dot inside to circle. 
    Once this is done, we have to choose $j - i$ labels from $[k + 1, n]$ as well as $j - i$ dots (outside boxes) to circle red. 
    Since the object must be primitive, there are $\binom{n - k}{j - i}\binom{n - (j - i + 2i)}{j - i}$ ways to do this.
\end{proof}

\section*{Acknowledgements}

The author is supported by the Göran Gustafsson Foundation and the Verg Foundation. 
The author thanks \rotatebox{2.75}{Lopezno} Vecchi for helpful distractions.

\bibliographystyle{abbrv}
\bibliography{refs}

\end{document}